\documentclass[reqno]{amsart}
\usepackage[english]{babel}
\usepackage{amscd,amssymb,amsmath,amsfonts,latexsym,amsthm}
\usepackage{inputenc}
\usepackage{graphicx}
\usepackage{booktabs}  
\usepackage{array}     
\usepackage{caption}      
\usepackage{float}        
\usepackage{hyperref,amsthm,amsmath,amssymb,color,multirow,graphicx}
\usepackage[fontsize=11pt]{scrextend}
\usepackage{placeins}

\newtheorem{thm}{Theorem}[section]

\newtheorem{lem}[thm]{Lemma}
\newtheorem{prop}[thm]{Proposition}

\theoremstyle{definition}
\newtheorem{defn}[thm]{Definition}

\numberwithin{equation}{section}

\newcommand{\be}{\begin{equation}}

\newcommand{\ee}{\end{equation}}

\begin{document}

\sloppy

\begin{center}
\textbf{\large  FORWARD AND INVERSE PROBLEMS FOR A TIME-FRACTIONAL PSEUDO-PARABOLIC EQUATION WITH VARIABLE COEFFICIENTS}\\

\textbf{ Ravshan Ashurov and Elbek Husanov}
\author[R. Ashurov and E. Husanov]
{Ravshan Ashurov and Elbek Husanov}

\address{\textcolor[rgb]{0.00,0.00,0.84}{Ravshan Ashurov:\newline 
V.~I.~Romanovskiy Institute of Mathematics Uzbekistan Academy of Sciences, University street 9, Tashkent, 100174, Uzbekistan. \newline
Central Asian University, 264 Milliy Bog Street, Tashkent, 111221, Uzbekistan.
}}
\email{\textcolor[rgb]{0.00,0.00,0.84}{ashurovr@gmail.com}}

\address{\textcolor[rgb]{0.00,0.00,0.84}{Elbek Husanov: \newline
V.~I.~Romanovskiy Institute of Mathematics Uzbekistan Academy of Sciences, University street 9, Tashkent, 100174, Uzbekistan.
}}
\email{\textcolor[rgb]{0.00,0.00,0.84}{elbekhusanov02@gmail.com}}

\end{center}

\textbf{Abstract.} In this work, forward and inverse problems for a time-fractional pseudo-parabolic equation $D_t^{\rho} [u(t) + \mu Au(t)] +
\sigma(t) Au(t) = r(t)g$ are investigated in a Hilbert space, where $A$ is an unbounded, positive, self-adjoint operator. According to the known papers, the forward problem has been studied only in the case $\sigma(t) = const$. The main novelty of
the forward problem in this work is that the model is further generalized and investigated for a time-dependent coefficient $\sigma(t)$. To determine the solution of the forward problem, the Fourier method is employed, and the global existence and uniqueness of the solution are proved. Moreover, when the operator $A$ is a second-order differential operator, a numerical scheme and an efficient computational algorithm are developed. The
inverse problem of determining a time-dependent source function is considered under the overdetermination condition of the form $F[u(t)] = \Phi(t)$. The functional $F$ is taken in a general form, and such an inverse problem has not been considered before. The global existence of the solution to the inverse problem is proved by applying Schauder's fixed point theorem, and its
uniqueness is established. Furthermore, several examples related to the operator $A$ and the functional $F$ are provided. 


\textbf{MSC (2010):}

\makeatletter
\renewcommand{\@evenhead}{\vbox{\thepage \hfil {\it R. Ashurov and E. Husanov }  \hrule }}
\renewcommand{\@oddhead}{\vbox{\hfill
{\it Inverse problem for the fractional-order heat equation with an unknown source term }\hfill
\thepage \hrule}} \makeatother

\label{firstpage}

\section{INTRODUCTION}

Let $H$ be a separable Hilbert space with an inner product $(\cdot,\cdot)$ and the associated norm $\|\cdot\|$, and let $A: D(A) \subset H \rightarrow H$ be an unbounded, positive, self-adjoint operator, where $D(A)$ denotes the domain of $A$. Suppose that the inverse operator $A^{-1}$ exists and is compact. Then there exists a sequence $\left\{\lambda_k\right\}_{k=1}^{\infty}$ and a corresponding orthonormal system $\left\{v_k\right\}_{k=1}^{\infty}$ such that
$$
A v_k=\lambda_k v_k, \qquad k \in \mathbb{N}.
$$

Here, $\lambda_k$ are the eigenvalues of the operator $A$ and $v_k$ are the associated eigenvectors. Without loss of generality, the eigenvalues can be arranged in a nondecreasing order:
$$
0 < \lambda_1 \le \lambda_2 \le \cdots \le \lambda_k \le \cdots,
\qquad \text{as } k \to \infty.
$$

Let $C((0, T) ; H)$ denote the set of continuous vector-valued functions $w(t)$ defined for $t \in(0, T)$ with values in $H$. For $w: \mathbb{R}_{+} \rightarrow H$ fractional integrals and derivatives are defined similarly to those for scalar functions (see, e.g., \cite{Lizama}). The Caputo fractional derivative of order $0 < \rho < 1$ is defined as (see, e.g., \cite{Pskhu}, p. 14)
$$
D_t^\rho w(t)=\frac{1}{\Gamma(1-\rho)} \int_0^t \frac{w^{\prime}(\xi)}{(t-\xi)^\rho} d \xi , \qquad t>0,
$$
provided the right-hand side exists. Here $\Gamma(\cdot)$ is Euler's gamma function.

In this paper, we consider the following time-fractional pseudo-parabolic equation
\begin{equation} \label{1.1}
    D_t^{\rho} [u(t) + \mu Au(t)] + \sigma(t) Au(t) = r(t)g, \qquad t \in (0,T],
\end{equation}
with the initial condition
\begin{equation} \label{1.2}
    u(0) = \phi,
\end{equation}
where $\phi, g \in H$, $\mu>0$ is a constant, $\sigma(t)$ and $r(t)$ are given scalar functions. 

If the functions $\sigma(t)$, $r(t)$ and the element $g$ are known, the problem \eqref{1.1}–\eqref{1.2} is called the \textit{forward problem}. 
On the other hand, if the coefficient $r(t)$ is unknown, the problem of determining the pair of functions $\{u(t); r(t)\}$ is called the \textit{inverse problem}. 
To solve the inverse problem, we impose the following additional condition:
\begin{equation} \label{1.3}
F[u(t)] = \Phi(t), \qquad t \in [0,T],
\end{equation}
where $F:H \to \mathbb{R}$ is a given bounded linear functional and $\Phi(t)$ is a given scalar function. When considering the inverse problem, we take $g \ne 0$.

If equation \eqref{1.1} has an integer-order time derivative, then in \cite{Chen} the linearized form of the energy equation for an isotropic material is expressed as follows:
\begin{equation} \label{1.4}
   c \frac{\partial}{\partial t} \left ( u - \mu \Delta u\right) - \sigma \Delta u = r, 
\end{equation}
where $c$, $\sigma$ and $\mu$ are constants, with $c$ denoting the specific heat, $\sigma$ the thermal conductivity, and $r$ the heat supplied (per unit volume) from the external world.

In \cite{Barenblat_Kochiva}, for equation \eqref{1.4} with coefficients $c=1$ and $r=0$, the fluid pressure in the fissures of fissured rocks was studied, where the coefficient $\mu$ represents a new characteristic specific to fissured rocks, and $\sigma$ denotes the piezo-conductivity coefficient of the fissured rock.

In pseudo-parabolic equations, inverse problems aimed at determining the coefficient and the right-hand side function by means of additional conditions have wide applications in applied science and engineering. The study of such inverse problems was first initiated in the 1980s by Rundell (see, \cite{Rundell}).

In recent years, time-fractional pseudo-parabolic equations involving Caputo-type fractional derivatives with $\rho \in (0,1)$ have been extensively studied and have attracted significant interest from researchers. Below, we briefly review some related works in this direction.

For the case $\mu>0$ and $\sigma(t)=1$ in \cite{Ruzhansky_psedo}, the forward and inverse problems for a time-fractional pseudo-parabolic equation with self-adjoint operators were studied in the Hilbert space in the following form:
$$
\begin{cases}
D_t^\rho[u(t)+L u(t)]+M u(t)=f(t), \qquad t \in (0,T]; \\
u(0)=\phi,
\end{cases}
$$
where $\phi \in H$ and $f(t)$ is a given function. In the inverse problem of determining the right-hand side, the unknown function is assumed to be independent of time. As an additional condition, one imposes
$$
 u(T)=\psi,
$$
where $\psi \in H$.

Furthermore, for time-fractional pseudo-parabolic equations, both the forward problem and the inverse problem of determining the kernel function were studied in \cite{Durdiev_Elmuradova}. The forward problem is formulated as follows:
$$
\begin{cases}
 D_t^\rho [u(x, t)- u_{x x}(x, t)]-u_{x x}(x, t)=\int_0^t k(t-\tau) u_{x x}(x, \tau) d \tau+f(x, t), \qquad (x,t) \in (0,1) \times (0,T];\\ 
u(x, 0)=\phi(x),   \qquad x \in [0,1];\\
u(0, t)=u(1, t)=0, \qquad t \in [0,T],
\end{cases}
$$
where $k(t)$, $\phi(x)$ and $f(x,t)$ are given functions. In the inverse problem, the kernel function $k(t)$ is considered to be unknown and is determined by means of the following additional condition:
$$
\int_0^1 \omega(x) u(x, t) d x=h(t),
$$
where $\omega(x)$ and $h(t)$ are given functions. 

For the case $\mu = 0$, equation \eqref{1.1} reduces to a time-fractional diffusion equation. In \cite{Zhang}–\cite{SerikbaevRuzhanskyTokmagambetov}, both the forward problem for this equation and inverse problems aimed at determining the coefficient $\sigma(t)$ were investigated. Moreover, \cite{SerikbaevRuzhanskyTokmagambetov2} considers an inverse problem concerning the identification of the function $r(t)$ on the right-hand side of the equation.

\section{PRELIMINARIES} 
In this section, we present some essential concepts and results that will be used in the study of the forward and inverse problems.

\subsection{The operator $A^{\delta}$ and its domain} 

Let $\delta$ be an arbitrary real number. We introduce the power of operator $A^{\delta}$ acting in $H$ according to the rule 
$$
A^{\delta}w = \sum_{k=1}^{\infty} \lambda_k^{\delta}  w_k v_k,
$$
where $w_k$ is the Fourier coefficients of a function $w \in H : w_k = (w,v_k)$. Obviously, the domain of
this operator has the form
$$
D(A^{\delta}) = \{w \in H: \sum_{k=1}^{\infty} \lambda_k^{2\delta}  |w_k|^2 < \infty  \}.
$$

For elements of $D(A^{\delta})$, we introduce the norm 
$$
\| w \|^2_{D(A^{\delta})} = \sum_{k=1}^{\infty} \lambda_k^{2\delta}  |w_k|^2 =  \| A^{\delta}w \|^2.
$$

\subsection{Mittag-Leffler function and some of its important properties} 

\begin{defn}\label{def2.1} (see, \cite{Dzhrbashyan}, p. 42) The two-parameter Mittag–Leffler function $E_{\rho ,\beta } (z)$ is defined as follows:
$$E_{\rho ,\beta } (z) = \sum _{k=0}^{\infty }\frac{z^{k}}{\Gamma \left(k\rho +\beta \right)}, \quad z \in {\mathbb C},$$
where $0 < \rho < 2$, $\beta \in \mathbb{C}$, and $\Gamma(\cdot)$ denotes the Gamma function.
\end{defn}

\begin{lem}\label{lem2.2} (see, \cite{SerikbaevRuzhanskyTokmagambetov}) 
If $\rho \in (0,1]$, $\lambda>0$ and $\beta \ge \rho$, then the Mittag-Leffler function $E_{\rho ,\beta}(-\lambda t^{\rho})$ satisfies the following inequality:
$$
0 \le E_{\rho ,\beta}(-\lambda t^{\rho}) \le \frac{1}{\Gamma(\beta)}, \quad t \ge 0.
$$
\end{lem}

\begin{prop}\label{prop2.3} (see, \cite{AshurovYusuf21})
Let $\lambda > 0$ and $\rho > 0$. Then the following integral identity holds:
$$
\int\limits_{0}^{t} \lambda \eta^{\rho-1} E_{\rho,\rho}(-\lambda \eta^{\rho}) \, d\eta = 1 - E_{\rho,1}(-\lambda t^{\rho}).
$$
\end{prop}

\subsection{Main theorems from functional analysis} 

\begin{thm} \label{thm2.4} \textbf{Generalized mean value theorem}. (see, \cite{Odibat}) Suppose that $f(t) \in C[a, b]$ and $D_a^\rho f(t) \in C(a, b]$, for $0<\rho \le 1$, then we have
$$
f(t)=f(a)+\frac{1}{\Gamma(\rho)}\left(D_a^\rho f\right)(\xi) \cdot(t-a)^\rho
$$
with $a \leqslant \xi \leqslant t$, for all $t \in(a, b]$.
\end{thm}

\begin{thm} \label{thm2.5} \textbf{Arzela-Ascoli theorem.} (see, \cite{KilbasSrivstavaTrujillo}, 68 p.) A necessary and sufficient condition that a subset of continuous functions $U$, which are defined on the closed interval $[a, b]$, be relatively compact in $C[a, b]$ is that this subset be uniformly bounded and equicontinuous.
\end{thm}

\begin{thm} \label{thm2.6} \textbf{Shauder's fixed point theorem.} (see, \cite{KilbasSrivstavaTrujillo}, 68 p.)  Let $U$ be a closed convex subset of $C[a, b]$, and let $A: U \rightarrow U$ be the map such that the set $\{A u: u \in U\}$ is relatively compact in $C[a, b]$. Then the operator $A$ has at least one fixed point $u^* \in U$ i.e. $A u^*=u^*$.
\end{thm}

\subsection{Gronwall type lemma.} 

\begin{lem}\label{lem2.7} (see, \cite{Henry}, Lemma 7.1.1) Suppose $c\geq 0$, $0<\rho<1$, and $z(t)$ is a non-negative function locally integrable on $[0, b)$ (for some $b \leq \infty$ ), and suppose $y(t)$ is non-negative and locally integrable on $[0, b)$ with
$$
y(t) \leq z(t)+c \int_0^t(t-\tau)^{\rho-1} y(\tau) d \tau, \quad \forall t \in[0, b).
$$

Then
$$
y(t) \leq z(t)+c \Gamma(\rho) \int_0^t \frac{d}{d \tau} E_{\rho, 1}\left(c \Gamma(\rho)(t-\tau)^\rho\right) z(\tau) d \tau, \qquad \forall t \in[0, b).
$$

If $z(t) \equiv z$ is a constant, then
$$
y(t) \leq z E_{\rho, 1}\left(c \Gamma(\rho) t^\rho\right), \forall t \in[0, b) .
$$
\end{lem}

\section{AUXILIARY CAUCHY PROBLEM} 

In this section, some important results on the Cauchy problem obtained by M. Ruzhansky et al. in \cite{SerikbaevRuzhanskyTokmagambetov} are presented.

Consider the Cauchy problem with $\lambda > 0$ and a given constant $y_0$:
\begin{equation}\label{3.1}
\left\{
\begin{aligned}
    & D_{t}^{\rho }y(t)+\lambda a(t)y(t)=f(t), \qquad t\in (0,T]; \\
    & y(0) = y_0,
\end{aligned}
\right.
\end{equation}
where $a(t)$ and $f(t)$ are given functions, and the order of the derivative satisfies $0 < \rho < 1$.

\begin{lem} \label{lem3.2} (see, \cite{SerikbaevRuzhanskyTokmagambetov}) There exists a unique solution $y(t) \in C[0, T]$ to the Cauchy problem \eqref{3.1}, which also satisfies $D_t^{\rho} y(t) \in C[0, T]$ , provided that the following conditions hold:\\
$ (1)\qquad a(t)\in C[0,T] \,\, and \,\, a(t) \ge a_1>0$; \\
$(2) \qquad f(t) \in C[0,T]$,\\
where $a_1$ is a positive constant.
\end{lem}

\begin{lem} (see, \cite{SerikbaevRuzhanskyTokmagambetov})
Under the conditions of Lemma \ref{lem3.2}, the solution $y(t)$ of the Cauchy problem \eqref{3.1} satisfies the estimate:
\begin{equation}\label{3.2}
       \left| y(t) \right| \le \left| y_0 \right| E_{\rho ,1} \left( -\lambda a_1 t^{\rho} \right) + \int_{0}^{t} \left| f(s) \right| (t-s)^{\rho -1} E_{\rho, \rho} \left( -\lambda a_1 (t-s)^{\rho} \right) \, ds.   
\end{equation}
\end{lem}

\section{INVESTIGATION OF THE FORWARD PROBLEM}
In this section, the forward problem corresponding to problem \eqref{1.1}–\eqref{1.2} is investigated, the existence and uniqueness of its solution are established.

\begin{defn} \label{def4.2} A function $u(t) \in {C([0,T];H)}$ with the properties $D_t^{\rho}u(t)$, $D_t^{\rho}(Au(t)) $, $Au(t) \in {C((0,T];H)}$ and satisfying conditions \eqref{1.1}–\eqref{1.2} is called the solution of the forward problem.    
\end{defn}

The main theorem of this section is presented below.

\begin{thm} \label{thm4.3} The problem \eqref{1.1}–\eqref{1.2} has a unique solution provided that the following conditions are satisfied:\\
$ (1)\qquad \sigma(t) \in C[0,T] \,\,and \,\, \sigma(t) \ge m_{\sigma}>0$; \\
$(2) \qquad r(t) \in C[0,T]$;\\
$(3) \qquad g \in H$;\\
$(4) \qquad \phi \in D(A)$,\\
where $m_{\sigma}$ is a positive constant.
\end{thm}

Assuming the existence of a solution to a problem \eqref{1.1} – \eqref{1.2} and noting that the eigenvectors ${v_k}$ form a complete set in $H$, the solution must be of the form
\begin{equation}\label{4.1}
   u(t) = \sum_{k=1}^{\infty } u_{k}(t)v_{k}, 
\end{equation}
where $u_k(t)$ are unknown functions.

Substituting \eqref{4.1} into \eqref{1.1}–\eqref{1.2} and using the conditions of Theorem~\ref{thm4.3}, the following problem for $u_k(t)$ is obtained:
\begin{equation}\label{4.2}
\left\{
\begin{aligned}
    & D_{t}^{\rho }{{u}_{k}}(t)+\frac{{{\lambda }_{k}}}{1+ \mu{{\lambda }_{k}}}\sigma (t){{u}_{k}}(t)=\frac{{{g}_{k}}}{1+\mu{{\lambda }_{k}}}r(t), \qquad  t \in (0,T]; \\
    & u_{k}(0) = \phi_{k},
\end{aligned}
\right.
\end{equation}
where $g_k$ and $\phi_k$ are the Fourier coefficients of the elements $g$ and $\phi$, respectively.

Let us rewrite the Cauchy problem \eqref{4.2} in the form:
\begin{equation}\label{4.3}
\left\{
\begin{aligned}
    & D_t^\rho u_k(t)+\frac{\lambda_k M_\sigma}{1+\mu \lambda_k} u_k(t)=\frac{\lambda_k}{1+\mu \lambda_k}\left(M_\sigma-\sigma(t)\right) u_k(t)+\frac{g_k}{1+\mu \lambda_k} r(t), \qquad  t \in (0,T]; \\
    & u_{k}(0) = \phi_{k},
\end{aligned}
\right.
\end{equation}
where ${{M}_{\sigma }}= \underset{t\in [0,T]}{\max } \sigma (t)$.

The problem \eqref{4.3} can be expressed in the form of the integral equation (see, \cite{KilbasSrivstavaTrujillo}, p. 230):
$$
u_k(t)=\phi_k E_{\rho, 1}\left(-\frac{\lambda_k M_\sigma}{1+\mu \lambda_k} t^\rho\right)+\frac{g_k}{1+\mu \lambda_k} \int_0^t r(s)(t-s)^{\rho-1} E_{\rho, \rho}\left(-\frac{\lambda_k M_\sigma}{1+\mu \lambda_k}(t-s)^\rho\right) ds
$$
\begin{equation}\label{4.4}
+\frac{\lambda_k}{1+\mu \lambda_k} \int_0^t\left(M_\sigma-\sigma(s)\right) u_k(s)(t-s)^{\rho-1} E_{\rho, \rho}\left(-\frac{\lambda_k M_\sigma}{1+\mu \lambda_k}(t-s)^\rho\right)ds. 
\end{equation}

The following lemma is established by applying Lemma~\ref{lem3.2} to the Cauchy problem \eqref{4.2}.

\begin{lem} \label{lem4.4} If the conditions of Theorem~\ref{thm4.3} are satisfied, then for each $k$, there exists a unique solution $u_k(t) \in C[0,T]$ of the Cauchy problem \eqref{4.2}. Moreover, this solution satisfies the following estimate:
$$
\left|u_k(t)\right| \leq\left|\phi_k\right| E_{\rho, 1}\left(-\frac{\lambda_k m_\sigma}{1+\mu \lambda_k}t^\rho\right)
$$
\begin{equation}\label{4.5}
 + \frac{\left|g_k\right|}{1+\mu \lambda_k} \int_0^t|r(s)|(t-s)^{\rho-1} E_{\rho, \rho}\left(-\frac{\lambda_k m_\sigma}{1+\mu \lambda_k}(t-s)^\rho\right) ds,  \qquad t\in [0,T].
\end{equation}
\end{lem}

\subsection{Existence and uniqueness of the solution of problem \eqref{1.1}–\eqref{1.2}} In this subsection, the existence and uniqueness of the solution for problem \eqref{1.1}–\eqref{1.2} are established.

Based on an estimate \eqref{4.5}, the regularity properties of the solution to problem \eqref{1.1}–\eqref{1.2} are derived. First, the norm of $Au(t)$ is estimated using Parseval’s identity:
$$
\left\| Au(t) \right\|_{H}^{2}=\sum\limits_{k=1}^{\infty }{\lambda _{k}^{2}{{\left| {{u}_{k}}(t) \right|}^{2}}} \le 2\sum\limits_{k=1}^{\infty }{\lambda _{k}^{2}{{\left| {{\phi }_{k}}{{E}_{\rho ,1}}\left( -\frac{{{\lambda }_{k}}{{m}_{\sigma }}}{1+\mu {{\lambda }_{k}}}{{t}^{\rho }} \right) \right|}^{2}}}
$$
$$
+2\sum\limits_{k=1}^{\infty }{\lambda _{k}^{2}{{\left| \frac{{{g}_{k}}}{1+\mu {{\lambda }_{k}}}\int_{0}^{t}{\left| r(s) \right|{{(t-s)}^{\rho -1}}{{E}_{\rho ,\rho }}\left( -\frac{{{\lambda }_{k}}{{m}_{\sigma }}}{1+\mu {{\lambda }_{k}}}{{(t-s)}^{\rho }} \right)ds} \right|}^{2}}}.
$$

Using Lemma \ref{lem2.2} and Proposition \ref{prop2.3}, an estimate for the norm is obtained:
$$
\left\| Au(t) \right\|_{H}^{2}\le 2\sum\limits_{k=1}^{\infty }{\lambda _{k}^{2}{{\left| {{\phi }_{k}} \right|}^{2}}}
$$
$$
+2\left\| r(t) \right\|_{C[0,T]}^{2}\sum\limits_{k=1}^{\infty }{{{\left| \frac{{{g}_{k}}}{{{m}_{\sigma }}}\int_{0}^{t}{\frac{{{\lambda }_{k}}{{m}_{\sigma }}}{1+\mu {{\lambda }_{k}}}{{(t-s)}^{\rho -1}}{{E}_{\rho ,\rho }}\left( -\frac{{{\lambda }_{k}}{{m}_{\sigma }}}{1+\mu {{\lambda }_{k}}}{{(t-s)}^{\rho }} \right)ds} \right|}^{2}}}
$$
$$
\leq 2\|\phi\|_{D(A)}^2+\frac{2}{m_\sigma^2}\|r(t)\|_{C[0, T]}^2\|g\|_H^2.
$$

It follows that
$$
\left\| Au(t) \right\|_{C((0,T];H)}^{2}\le  2\|\phi\|_{D(A)}^2+\frac{2}{m_\sigma^2}\|r(t)\|_{C[0, T]}^2\|g\|_H^2.
$$
Consequently, $Au(t) \in C((0,T];H)$ and hence $u(t) \in C([0,T];H)$.

Second, the problem \eqref{4.2} together with estimate \eqref{4.5} allows the norm of the function $D_{t}^{\rho }u(t)$ to be expressed via Parseval’s identity:
$$
\left\| D_{t}^{\rho }u(t) \right\|_{H}^{2}=\sum\limits_{k=1}^{\infty }{{{\left| D_{t}^{\rho }{{u}_{k}}(t) \right|}^{2}}}=\sum\limits_{k=1}^{\infty }{{{\left| \frac{{{g}_{k}}}{1+\mu {{\lambda }_{k}}}r(t)-\frac{{{\lambda }_{k}}}{1+\mu {{\lambda }_{k}}}\sigma (t){{u}_{k}}(t) \right|}^{2}}}
$$
$$
\le 2\sum\limits_{k=1}^{\infty }{{{\left| \frac{{{g}_{k}}}{1+\mu {{\lambda }_{k}}}r(t) \right|}^{2}}}+2\sum\limits_{k=1}^{\infty }{{{\left| \frac{{{\lambda }_{k}}}{1+\mu {{\lambda }_{k}}}\sigma (t){{u}_{k}}(t) \right|}^{2}}}\le 2\left\| r(t) \right\|_{C[0,T]}^{2}\sum\limits_{k=1}^{\infty }{{{\left| {{g}_{k}} \right|}^{2}}}+2M_{\sigma }^{2}\sum\limits_{k=1}^{\infty }{{{\left| {{\lambda }_{k}}{{u}_{k}}(t) \right|}^{2}}}.
$$

Thus,
$$
\left\| D_{t}^{\rho }u(t) \right\|_{C((0,T];H)}^{2}\le 2\left\| r(t) \right\|_{C[0,T]}^{2}\left\| g \right\|_{H}^{2}+2M_{\sigma }^{2}\left\| Au(t) \right\|_{C((0,T];H)}^{2}.
$$

According to conditions (2)–(3) of Theorem~\ref{thm4.3} and the fact that $Au(t) \in C((0,T];H)$, it can be seen that  $D_{t}^{\rho }u(t)\in C((0,T];H)$.

Finally, an estimate for $D_{t}^{\rho }(Au(t))$ is obtained. From equation \eqref{1.1}, one has
$$
\mu D_{t}^{\rho }(Au(t))=r(t)g-D_{t}^{\rho }u(t)-\sigma (t)Au(t), \qquad t \in (0,T],
$$
which allows the norm to be estimated accordingly:
$$
\mu {{\left\| D_{t}^{\rho }(Au(t)) \right\|}_{C((0,T];H)}}\le {{\left\| r(t) \right\|}_{C[0,T]}}{{\left\| g \right\|}_{H}}+{{\left\| D_{t}^{\rho }u(t) \right\|}_{C((0,T];H)}}+M_{\sigma}{{\left\| Au(t) \right\|}_{C((0,T];H)}}.
$$

In view of the results obtained above, it follows that $D_{t}^{\rho }(Au(t)) \in C((0,T];H)$.

Let us investigate the uniqueness of the solution to problem \eqref{1.1} - \eqref{1.2}. Suppose we have two solutions: $\bar{u}(t)$, $\tilde{u}(t)$ and set $u(t) = \bar{u}(t) - \tilde{u}(t)$. Then, we have 
\begin{equation} \label{4.6}
    D_t^{\rho} [u(t) + \mu Au(t)] + \sigma(t) Au(t) = 0, \qquad t \in (0,T],
\end{equation}
with the initial condition 
\begin{equation} \label{4.7}
    u(0) = 0.
\end{equation}

Let $u(t)$ be a solution to this problem and $u_k(t)=(u(t),v_k)$. Then, by virtue of equation \eqref{4.6} and the self-adjointness of operator $A$,
$$
D_{t}^{\rho }{{u}_{k}}(t)=(D_{t}^{\rho }u(t),{{v}_{k}})=-\mu (AD_{t}^{\rho }u(t),{{v}_{k}})-\sigma(t)(Au(t),{{v}_{k}})
$$
$$
=-\mu (D_{t}^{\rho }u(t),A{{v}_{k}})-\sigma(t)(u(t),A{{v}_{k}})=-\mu {{\lambda }_{k}}D_{t}^{\rho }{{u}_{k}}(t)-{{\lambda }_{k}}\sigma(t){{u}_{k}}(t).
$$

From the above equality and \eqref{4.7}, the Cauchy problem is obtained:
\begin{equation}\label{4.8}
\left\{
\begin{aligned}
    & D_{t}^{\rho }{{u}_{k}}(t)+\frac{{{\lambda }_{k}}}{1+ \mu{{\lambda }_{k}}}\sigma (t){{u}_{k}}(t)=0, \qquad t \in (0,T]; \\
    & u_{k}(0) = 0.
\end{aligned} \right.
\end{equation}

From estimate \eqref{4.5}, we have $\left| u_{k}(t) \right| \le 0$, which implies that $u_{k}(t)=0$ for all $k$. Consequently, by \eqref{4.1}, it follows that $u(t)=0$. Hence, the solution of the forward problem is unique.

\section{INVESTIGATION OF THE INVERSE PROBLEM} In this section, the inverse problem \eqref{1.1}–\eqref{1.3} is investigated, which consists in determining the pair of functions $\{u(t); r(t)\}$.

\begin{defn}\label{def5.1} A pair of functions $\{u(t);r(t)\}$ with the properties  $u(t)\in C([0,T];H)$,  $Au(t)$, $D_t^{\rho}u(t)$, $D_t^{\rho}(Au(t)) \in C((0,T];H)$, $r(t) \in C[0,T]$ and satisfying conditions \eqref{1.1}-\eqref{1.3} is called the solution of the inverse problem.  
\end{defn}

The main result of this section is stated below.

\begin{thm} \label{thm5.2} The inverse problem \eqref{1.1}–\eqref{1.3} admits a unique solution $\{u(t); r(t)\}$ provided that the following conditions are satisfied:\\
$(1) \qquad \Phi(t) \in C^{1}[0,T]$;\\
$(2) \qquad  Let\,\, us \,\, choose \,\, \gamma \in \mathbb{R} \,\, so \,\, that  \,\, \left\{\frac{F\left[v_k\right]}{\lambda_k^\gamma}\right\}_{k \in \mathbb{N}} \in l^2 $;\\
$(3) \qquad \phi \in D(A^{\gamma^*}), \,\,\, where \,\, \gamma^* = \max \{1, \gamma\}$;\\
$(4) \qquad g \in D(A^\gamma) \,\,\, with \,\,\, F[g] \ne 0 \,\,\, and \,\,\, F[{{(I+\mu A)}^{-1}}g] \ne 0, $\\
where $v_k$ denotes an eigenfunction of the operator $A$.
\end{thm}

To illustrate conditions (2) and (4) of Theorem~\ref{thm5.2}, we present some examples corresponding to different choices of the functional $F$ defined in \eqref{1.3} and to the element $g$.

\medskip
\noindent \textbf{Example.} 
Let $H$ be $L^2(0,1)$ and $Au=-u_{x x}$, $ x \in(0,1)$, with homogeneous Dirichlet boundary condition. Then the eigenvalues of $A$ are $\lambda_k = (\pi k)^2$ and the corresponding eigenfunctions are $v_k(x) = \sqrt{2}\sin(k\pi x)$ for all $k \in \mathbb{N}$.

\medskip
\noindent \textbf{Case 1.}
Let 
$$
F[u(x,t)] := u(x_0,t), \qquad x_0\in(0,1).
$$

Then $F[v_k(x)]=\sqrt{2}\sin(k\pi x_0)$. Hence,
$$
\sum_{k=1}^{\infty}\left|\frac{F\left[v_k(x)\right]}{\lambda_k^\gamma}\right|^2 \leq \sum_{k=1}^{\infty}\left|\frac{\sqrt{2} \sin \left(k \pi x_0\right)}{(k \pi)^{2 \gamma}}\right|^2 \leq \frac{2}{\pi^{4 \gamma}} \sum_{k=1}^{\infty} \frac{1}{k^{4 \gamma}}.
$$

Therefore, for any \(\gamma>\frac14\), it holds that $\left\{\frac{F[v_k(x)]}{\lambda_k^\gamma}\right\}_{k\in\mathbb N}\in \ell^2$.

\medskip
\noindent \textbf{Case 2.}
Let
$$
F[u(x,t)] := {{\left. u_x(x,t) \right|}_{x=1}}.
$$

Then
$$
F[v_k(x)]=\sqrt{2}k\pi\cos(k\pi)=\begin{cases}
-\sqrt{2}k\pi, & k=2n-1,\\[2mm]
\sqrt{2}k\pi, & k=2n,
\end{cases}
\qquad n\in\mathbb N.
$$

For any $\gamma \ge1$, it holds that $\left\{\frac{F[v_k(x)]}{\lambda^{\gamma}_k}\right\}_{k\in\mathbb N}\in \ell^2$.

\medskip
\noindent \textbf{Case 3.} 
Let
$$
F[u(x,t)] := \int_0^1 u(x,t)\,dx.
$$

Then
$$
F[v_k(x)]
=\int_0^1 \sqrt{2}\sin(k\pi x)\,dx
=\frac{\sqrt{2}\bigl(1+(-1)^{k+1}\bigr)}{k\pi}
=
\begin{cases}
\dfrac{2\sqrt{2}}{k\pi}, & k=2n-1,\\[2mm]
0, & k=2n, 
\end{cases}
\qquad n\in\mathbb N.
$$

For any $\gamma \ge 0$, it holds that $\left\{\frac{F[v_k(x)]}{\lambda^{\gamma}_k}\right\}_{k\in\mathbb N}\in \ell^2$.

Let us take $\gamma=0$ for Case 3. Then there exist functions $h \in H^2(0,1) \cap H_0^1(0,1)$ such that the function $g(x)$ can be represented in the form
$$
    g(x)=h(x)-\mu h''(x), \qquad 0<x<1.
$$

To verify condition (4) of Theorem~\ref{thm5.2}, we choose $h(x)=\sin(\pi x)$. Then $g(x)=(1+\mu\pi^2)\sin(\pi x)$. Since
$$
F[h(x)]=\int_0^1 \sin(\pi x)\,dx=\frac{2}{\pi}, \qquad F[g(x)]=\int_0^1 (1+\mu\pi^2)\sin(\pi x)\,dx=(1+\mu\pi^2)\frac{2}{\pi}.
$$

It follows that condition (4) of Theorem~\ref{thm5.2} is satisfied.

To determine the function $r(t)$, we apply the functional $F$ from condition \eqref{1.3} to equation \eqref{1.1}. This yields:
\begin{equation} \label{5.1}
 r(t)=\frac{F\left[D_t^\rho u(t)\right]+\mu F\left[D_t^\rho(A u(t))\right]+\sigma(t) F[A u(t)]}{F[g]},
\end{equation}
where $F[g] \neq 0$ according to condition (4) of Theorem~\ref{thm5.2}.

All actions of the functional $F$ in equation \eqref{5.1} are examined using the series representation in \eqref{4.1}. Thus, the following relations are obtained:
\begin{equation}\label{5.2}
  F\left[D_t^\rho u(t)\right]=\sum_{k=1}^{\infty} D_t^\rho u_k(t) F\left[v_k\right]=D_t^\rho \Phi(t);  
\end{equation}
\begin{equation}\label{5.3}
F\left[D_t^\rho(A u(t))\right]=\sum_{k=1}^{\infty} \lambda_k D_t^\rho u_k(t) F\left[v_k\right]
\end{equation}
and
\begin{equation}\label{5.4}
F[A u(t)]=\sum_{k=1}^{\infty} \lambda_k u_k(t) F\left[v_k\right].
\end{equation}

Using \eqref{5.2}, \eqref{5.3} and \eqref{5.4}, equation \eqref{5.1} can be represented as

\begin{equation}\label{5.5}
r(t)=\frac{D_t^\rho \Phi(t)}{F[g]}+\frac{\mu}{F[g]} \sum_{k=1}^{\infty} \lambda_k D_t^\rho u_k(t) F\left[v_k\right]+\frac{\sigma(t)}{F[g]} \sum_{k=1}^{\infty} \lambda_k u_k(t) F\left[v_k\right].
\end{equation}

From \eqref{4.2}, the term $\mu \lambda_k D_t^\rho u_k(t)$ is given by
$$
\mu \lambda_k D_t^\rho u_k(t)=-\frac{\mu \lambda_k^2}{1+\mu \lambda_k} \sigma(t) u_k(t)+\frac{\mu \lambda_k}{1+\mu \lambda_k} g_k r(t).
$$

Substituting this expression into \eqref{5.5} gives the following relation for $r(t)$:
$$
r(t)=\frac{D_t^\rho \Phi(t)}{F[g]}+\frac{1}{F[g]} \sum_{k=1}^{\infty}\left[-\frac{\mu \lambda_k^2}{1+\mu \lambda_k} \sigma(t) u_k(t)+\frac{\mu \lambda_k}{1+\mu \lambda_k} g_k r(t)\right] F\left[v_k\right]+\frac{\sigma(t)}{F[g]} \sum_{k=1}^{\infty} \lambda_k u_k(t) F\left[v_k\right].
$$

Rearranging this expression, one obtains
\begin{equation} \label{5.6}
\left[ F[g]-\sum\limits_{k=1}^{\infty }{\frac{\mu {{\lambda }_{k}}}{1+\mu {{\lambda }_{k}}}{{g}_{k}}F\left[ {{v}_{k}} \right]} \right]r(t)=D_{t}^{\rho }\Phi (t)+\sigma (t)\sum\limits_{k=1}^{\infty }{\frac{{{\lambda }_{k}}}{1+\mu {{\lambda }_{k}}}}{{u}_{k}}(t)F\left[ {{v}_{k}} \right].
\end{equation}

The bracketed term on the left-hand side of \eqref{5.6} is given by
$$
F[g]-\sum\limits_{k=1}^{\infty }{\frac{\mu {{\lambda }_{k}}}{1+\mu {{\lambda }_{k}}}{{g}_{k}}F\left[ {{v}_{k}} \right]}=\sum\limits_{k=1}^{\infty }{\frac{1}{1+\mu {{\lambda }_{k}}}{{g}_{k}}F\left[ {{v}_{k}} \right]}=F[{{(I+\mu A)}^{-1}}g].
$$

By condition (4) of Theorem~\ref{thm5.2}, it holds that $F[(I+\mu A)^{-1}g] \ne 0$.

From the series representation in \eqref{5.6}, the explicit form of the function $r(t)$ is expressed:
\begin{equation}\label{5.7}
r(t)=\frac{D_{t}^{\rho }\Phi(t)}{F[{{(I+\mu A)}^{-1}}g]}+\frac{\sigma (t)}{F[{{(I+\mu A)}^{-1}}g]}\sum\limits_{k=1}^{\infty }{\frac{{{\lambda }_{k}}}{1+\mu {{\lambda }_{k}}}{{u}_{k}}(t)F\left[ {{v}_{k}} \right]},
\end{equation}
where $u_k(t)$ denotes the solution of the integral equation \eqref{4.4}, which depends on the function $r(t)$. 

\subsection{Existence of solution for the inverse problem} In this subsection, the existence of a solution to the inverse problem \eqref{1.1}-\eqref{1.3} is established by employing Schauder's fixed point theorem.

Let a priori estimate for the function $r(t)$ be established using equation \eqref{5.7}:
$$
\left| r(t) \right| \leq \left| \frac{D_{t}^{\rho }\Phi (t)}{F[{{(I+\mu A)}^{-1}}g]} \right|+\left| \frac{\sigma (t)}{\mu F[{{(I+\mu A)}^{-1}}g]} \right|\sum\limits_{k=1}^{\infty }{\frac{\mu {{\lambda }_{k}}}{1+\mu {{\lambda }_{k}}}\left| {{u}_{k}}(t) \right|\left| F\left[ {{v}_{k}} \right] \right|}.
$$

By applying the inequality $0<\frac{\mu {{\lambda }_{k}}}{1+\mu {{\lambda }_{k}}}<1$ together with estimate \eqref{4.5}, the following estimate follows:
$$
\left| r(t) \right| \leq \left| \frac{D_{t}^{\rho }\Phi (t)}{F[{{(I+\mu A)}^{-1}}g]} \right|+\frac{{{M}_{\sigma }}}{\mu \left| F[{{(I+\mu A)}^{-1}}g] \right|}\sum\limits_{k=1}^{\infty }{\left| {{u}_{k}}(t) \right|\left| F\left[ {{v}_{k}} \right] \right|}\le\left| \frac{D_{t}^{\rho }\Phi (t)}{F[{{(I+\mu A)}^{-1}}g]} \right|+I_1+I_2,
$$
where
$$
I_1=\frac{{{M}_{\sigma }}}{\mu \left| F[{{(I+\mu A)}^{-1}}g] \right|}\sum\limits_{k=1}^{\infty }{\left| {{\phi }_{k}} \right|{{E}_{\rho ,1}}\left( -\frac{{{\lambda }_{k}}{{m}_{\sigma }}}{1+\mu {{\lambda }_{k}}}{{t}^{\rho }} \right)\left| F\left[ {{v}_{k}} \right] \right|}
$$
and
$$
{{I}_{2}}=\frac{{{M}_{\sigma }}}{\mu \left| F[{{(I+\mu A)}^{-1}}g] \right|}\sum\limits_{k=1}^{\infty }{\frac{\left| {{g}_{k}} \right|}{1+\mu {{\lambda }_{k}}}\int\limits_{0}^{t}{\left| r(s) \right|{{(t-s)}^{\rho -1}}{{E}_{\rho ,\rho }}\left( -\frac{{{\lambda }_{k}}{{m}_{\sigma }}}{1+\mu {{\lambda }_{k}}}{{(t-s)}^{\rho }} \right)ds}}\left| F\left[ {{v}_{k}} \right] \right|.
$$

To obtain the desired bounds, we consider the sums $I_1$ and $I_2$ individually. First, $I_1$ is estimated using the Cauchy–Schwarz inequality together with Lemma~\ref{lem2.2}:
\begin{equation}\label{5.8}
{{I}_{1}}\le \frac{{{M}_{\sigma }}}{\mu \left| F[{{(I+\mu A)}^{-1}}g] \right|}\sqrt{\sum\limits_{k=1}^{\infty }{{{\left| \lambda _{k}^{\gamma  }{{\phi }_{k}} \right|}^{2}}}}\sqrt{\sum\limits_{k=1}^{\infty }{{{\left| \frac{F\left[ {{v}_{k}} \right]}{\lambda _{k}^{\gamma  }} \right|}^{2}}}}\le \frac{{{C}_{F}}{{M}_{\sigma }}}{\mu \left| F[{{(I+\mu A)}^{-1}}g] \right|}{{\left\| \phi  \right\|}_{D({{A}^{\gamma  }})}},
\end{equation}
where ${{C}_{F}}=\sqrt{\sum\limits_{k=1}^{\infty }{{{\left| \frac{F\left[ {{v}_{k}} \right]}{\lambda _{k}^{\gamma  }} \right|}^{2}}}}$.

Second, the sum $I_2$ is estimated by applying Lemma~\ref{lem2.2}, taking into account that $\frac{1}{1+\mu \lambda_k}<1$:
$$
{{I}_{2}}\le \frac{{{M}_{\sigma }}}{\mu \Gamma (\rho )\left| F[{{(I+\mu A)}^{-1}}g] \right|}\int\limits_{0}^{t}{|r(s)|{{(t-s)}^{\rho -1}}ds}\sum\limits_{k=1}^{\infty }{\left| {{g}_{k}} \right|\left| F\left[ {{v}_{k}} \right] \right|}.
$$

By applying the Cauchy–Schwarz inequality, $I_2$ can be estimated as follows:
$$
{{I}_{2}}\le \frac{{{M}_{\sigma }}}{\mu \Gamma (\rho )\left| F[{{(I+\mu A)}^{-1}}g] \right|}\sqrt{\sum\limits_{k=1}^{\infty }{{{\left| \lambda _{k}^{\gamma  }{{g}_{k}} \right|}^{2}}}}\sqrt{\sum\limits_{k=1}^{\infty }{{{\left| \frac{F\left[ {{v}_{k}} \right]}{\lambda _{k}^{\gamma  }} \right|}^{2}}}}\int\limits_{0}^{t}{|r(s)|{{(t-s)}^{\rho -1}}ds}
$$
\begin{equation}\label{5.9}
\le \frac{{{C}_{F}}{{M}_{\sigma }}}{\mu \Gamma (\rho )\left| F[{{(I+\mu A)}^{-1}}g] \right|}{{\left\| g \right\|}_{D(A^{\gamma })}}\int\limits_{0}^{t}{|r(s)|{{(t-s)}^{\rho -1}}ds}.
\end{equation}

An inequality for the function $r(t)$ can be obtained from the estimates \eqref{5.8} and \eqref{5.9}:
$$
\left| r(t) \right| \le\left| \frac{D_{t}^{\rho }\Phi (t)}{F[{{(I+\mu A)}^{-1}}g]} \right|+\frac{{{C}_{F}}{{M}_{\sigma }}}{\mu \left| F[{{(I+\mu A)}^{-1}}g] \right|}{{\left\| \phi  \right\|}_{D(A^{\gamma })}}
$$
$$
+\frac{{{C}_{F}}{{M}_{\sigma }}}{\mu \Gamma (\rho )\left| F[{{(I+\mu A)}^{-1}}g] \right|}{{\left\| g \right\|}_{D(A^{\gamma })}}\int\limits_{0}^{t}{|r(s)|{{(t-s)}^{\rho -1}}ds}.
$$

Applying Lemma~\ref{lem2.7} to the above inequality, the following estimate is obtained:
$$
\left| r(t) \right|\le \left( \frac{1}{\left| F[{{(I+\mu A)}^{-1}}g] \right|}{{\left\| D_{t}^{\rho }\Phi (t) \right\|}_{C[0,T]}}+\frac{{{C}_{F}}{{M}_{\sigma }}}{\mu \left| F[{{(I+\mu A)}^{-1}}g] \right|}{{\left\| \phi  \right\|}_{D(A^{\gamma })}} \right)
$$
$$
\times {{E}_{\rho ,1}}\left( \frac{{{C}_{F}}{{M}_{\sigma }}}{\mu \left| F[{{(I+\mu A)}^{-1}}g] \right|}{{\left\| g \right\|}_{D(A^{\gamma })}}{{t}^{\rho }} \right).
$$

According to the properties of the Mittag–Leffler function, we see that
\begin{equation}\label{5.10}
    \|r(t)\|_{C[0,T]} \leq C_1,
\end{equation}
where
$$
C_1=\left( \frac{1}{\left| F[{{(I+\mu A)}^{-1}}g] \right|}{{\left\| D_{t}^{\rho }\Phi (t) \right\|}_{C[0,T]}}+\frac{{{C}_{F}}{{M}_{\sigma }}}{\mu \left| F[{{(I+\mu A)}^{-1}}g] \right|}{{\left\| \phi  \right\|}_{D(A^{\gamma })}} \right)
$$
$$
\times {{E}_{\rho ,1}}\left( \frac{{{C}_{F}}{{M}_{\sigma }}{{T}^{\rho }}}{\mu \left| F[{{(I+\mu A)}^{-1}}g] \right|}{{\left\| g \right\|}_{D(A^{\gamma })}} \right).
$$

If the inverse problem \eqref{1.1}–\eqref{1.3} has solutions, then the set containing these solutions is introduced:
$$
U := \{ r \in C[0,T] : \|r(t)\|_{C[0,T]} \leq C_1 \}.
$$

It is obvious that $U$ is convex.

Let us introduce an operator $B$, which is essential for establishing the existence of a solution to the inverse problem. The operator $B$ is defined as follows:
\begin{equation}\label{5.11}
B[r](t)=\frac{D_{t}^{\rho }\Phi(t)}{F[{{(I+\mu A)}^{-1}}g]}+\frac{\sigma (t)}{F[{{(I+\mu A)}^{-1}}g]}\sum\limits_{k=1}^{\infty }{\frac{{{\lambda }_{k}}}{1+\mu {{\lambda }_{k}}}{{u}_{k}}(t)F\left[ {{v}_{k}} \right]}.
\end{equation}

Here, $u_k(t)$ denotes the solution of the integral equation \eqref{4.4} associated with the function $r(t)$.

Equation \eqref{5.7} can be rewritten in a more convenient form with the help of equation \eqref{5.11}:
\begin{equation}\label{5.12}
  r(t)=B[r](t).  
\end{equation}

Let us introduce the new set
$$
\mathcal{K} := \{ B[r](t) : r(t) \in U \}.
$$

Let us now consider the necessary and sufficient conditions for the set $\mathcal{K}$ to be relatively compact, for which we employ the \textit{Arzela–Ascoli Theorem}~\ref{thm2.5}.

\begin{defn} \label{def5.3} (see, \cite{Kolmogorov}, p. 54) A set $\mathcal{K}$ is called \textit{uniformly bounded} if there exists a constant $C>0$ such that $\|B[r](t)\|_{C[0,T]} \leqq C$ for every $B[r](t) \in \mathcal{K}$.
\end{defn}

\begin{defn} \label{def5.4} (see, \cite{Kolmogorov}, p. 54) A set $\mathcal{K}$ is called \textit{equicontinuous} if, for every $\varepsilon>0$, there exists some $\delta>0$ such that, for all $B[r](t) \in \mathcal{K}$ and all $t_1, t_2 \in[0,T]$ with $\left|t_1-t_2\right|<\delta$, we have $\left|B[r]\left(t_1\right)-B[r]\left(t_2\right)\right|<\varepsilon$.   
\end{defn}

\begin{lem} \label{lem5.5} The set $\mathcal{K}$ is relatively compact in $C[0,T]$.
\end{lem}
\begin{proof} By the \textit{Arzela–Ascoli Theorem}~\ref{thm2.5}, it is sufficient to show that the set $\mathcal{K}$ is uniformly bounded and equicontinuous. Firstly, the boundedness of the set $\mathcal{K}$ is established using \eqref{5.10} and \eqref{5.12}. This implies that
$$
\|B[r](t)\|_{C[0,T]} \leq C_1.
$$

By Definition~\ref{def5.3}, the set $\mathcal{K}$ is uniformly bounded.

Secondly, we prove that the set $\mathcal{K}$ is equicontinuous. For convenience, assume that $t_1<t_2$, in view of \eqref{5.11} the following difference is estimated:
$$
\left| B[r]\left( {{t}_{1}} \right)-B[r]\left( {{t}_{2}} \right) \right|\le \frac{1}{\left| F[{{(I+\mu A)}^{-1}}g] \right|}\left| D_{t}^{\rho }\Phi \left( {{t}_{1}} \right)-D_{t}^{\rho }\Phi \left( {{t}_{2}} \right) \right|
$$
$$
+\frac{\sigma \left( {{t}_{1}} \right)}{\left| F[{{(I+\mu A)}^{-1}}g] \right|}\sum\limits_{k=1}^{\infty }{\frac{{{\lambda }_{k}}}{1+\mu {{\lambda }_{k}}}}\left| {{u}_{k}}\left( {{t}_{1}} \right)-{{u}_{k}}\left( {{t}_{2}} \right) \right|\left| F\left[ {{v}_{k}} \right] \right|
$$
$$
+\frac{\left| \sigma \left( {{t}_{1}} \right)-\sigma \left( {{t}_{2}} \right) \right|}{\left| F[{{(I+\mu A)}^{-1}}g] \right|}\sum\limits_{k=1}^{\infty }{\frac{{{\lambda }_{k}}}{1+\mu {{\lambda }_{k}}}}\left| {{u}_{k}}\left( {{t}_{2}} \right) \right|\left| F\left[ {{v}_{k}} \right] \right|={{J}_{1}}+{{J}_{2}}+{{J}_{3}},
$$
where 
$$
J_1 =  \frac{1}{\left| F[{{(I+\mu A)}^{-1}}g] \right|}\left| D_{t}^{\rho }\Phi \left( {{t}_{1}} \right)-D_{t}^{\rho }\Phi \left( {{t}_{2}} \right) \right|,
$$
$$
J_2 = \frac{\sigma \left( {{t}_{1}} \right)}{\left| F[{{(I+\mu A)}^{-1}}g] \right|}\sum\limits_{k=1}^{\infty }{\frac{{{\lambda }_{k}}}{1+\mu {{\lambda }_{k}}}}\left| {{u}_{k}}\left( {{t}_{1}} \right)-{{u}_{k}}\left( {{t}_{2}} \right) \right|\left| F\left[ {{v}_{k}} \right] \right|
$$
and
$$
J_3 = \frac{\left| \sigma \left( {{t}_{1}} \right)-\sigma \left( {{t}_{2}} \right) \right|}{\left| F[{{(I+\mu A)}^{-1}}g] \right|}\sum\limits_{k=1}^{\infty }{\frac{{{\lambda }_{k}}}{1+\mu {{\lambda }_{k}}}}\left| {{u}_{k}}\left( {{t}_{2}} \right) \right|\left| F\left[ {{v}_{k}} \right] \right|.
$$

Let us first estimate $J_1$. By condition (1) of Theorem~\ref{thm5.2}, we have $D_t^{\rho}\Phi(t) \in C[0,T]$. Consequently, for any $\varepsilon_1>0$, there exists $\delta_1>0$ such that for all $t_1, t_2 \in [0,T]$, if $|t_1 - t_2| < \delta_1$, the following inequality holds:
$$
\big|D_t^{\rho}\Phi(t_1) - D_t^{\rho}\Phi(t_2)\big| < \varepsilon_1.
$$

By setting $\varepsilon_1=\left| F[{{(I+\mu A)}^{-1}}g] \right|\frac{\varepsilon}{3}$, there exists a corresponding $\delta_1=\delta(\varepsilon_1)$, which implies that 
\begin{equation}\label{5.13}
   J_1 <\frac{\varepsilon}{3}. 
\end{equation}

Secondly, $J_2$ is estimated by taking into account condition (1) of Theorem ~\ref{thm4.3} and the inequality $\frac{\mu \lambda_k}{1+\mu \lambda_k}<1$ which then leads to
$$
{{J}_{2}}\le \frac{{{M}_{\sigma }}}{\mu \left| F[{{(I+\mu A)}^{-1}}g] \right|}\sum\limits_{k=1}^{\infty }{\left| {{u}_{k}}\left( {{t}_{1}} \right)-{{u}_{k}}\left( {{t}_{2}} \right) \right|\left| F\left[ {{v}_{k}} \right] \right|}.
$$

To simplify this inequality, Theorem~\ref{thm2.4} is applied, which yields
\begin{equation} \label{5.14}
 \left|u_k\left(t_1\right)-u_k\left(t_2\right)\right| = \frac{1}{\Gamma(\rho)}\left|_{t_1}D_t^\rho u_k(\xi)\right|\left|t_1-t_2\right|^\rho,   
\end{equation}
where 
$$
_{t_1}D_{t}^{\rho }{{u}_{k}}(\xi )=\frac{1}{\Gamma (1-\rho )}{{\left. \int\limits_{{{t}_{1}}}^{t}{\frac{u'_k(s)}{{{(t-s)}^{\rho }}}}ds \right|}_{t=\xi}} \qquad \text{for some} \quad \xi \in [t_1,t_2].
$$

Based on the fundamental properties of the integral, the  inequality holds:
$$
_{{{t}_{1}}}D_{t}^{\rho }{{u}_{k}}(\xi )=\frac{1}{\Gamma (1-\rho )}{{\left. \int\limits_{0}^{t}{\frac{{{{{u}'}}_{k}}(s)}{{{(t-s)}^{\rho }}}}ds \right|}_{t=\xi }}-\frac{1}{\Gamma (1-\rho )}\int\limits_{0}^{{{t}_{1}}}{\frac{{{{{u}'}}_{k}}(s)}{{{(t-s)}^{\rho }}}}ds=D_{t}^{\rho }{{u}_{k}}(\xi )-D_{t}^{\rho }{{u}_{k}}({{t}_{1}}).
$$

Let $r(t) \in U$. Then, we derive the following upper bound for $_{t_1}D_t^{\rho} u_k(\xi)$:
\begin{equation}\label{5.15}
 \left| _{{{t}_{1}}}D_{t}^{\rho }{{u}_{k}}(\xi ) \right|=\left| D_{t}^{\rho }{{u}_{k}}(\xi )-D_{t}^{\rho }{{u}_{k}}({{t}_{1}})\, \right|\le \frac{{{C}_{1}}}{1+\mu {{\lambda }_{k}}}\left| {{g}_{k}} \right|+\frac{2{{M}_{\sigma }}{{\lambda }_{k}}}{1+\mu {{\lambda }_{k}}}{{\left\| {{u}_{k}}(t) \right\|}_{C[0,T]}}. 
\end{equation}

Let us simplify the expression using estimate \eqref{4.5} and Lemma~\eqref{lem2.2}, namely:
$$
\left| {{u}_{k}}(t) \right|\le \left| {{\phi }_{k}} \right|{{E}_{\rho ,1}}\left( -\frac{{{\lambda }_{k}}{{m}_{\sigma }}}{1+\mu {{\lambda }_{k}}}{{t}^{\rho }} \right)
$$
$$
+\frac{\left| {{g}_{k}} \right|}{{{\lambda }_{k}}{{m}_{\sigma }}}{{\left\| r(t) \right\|}_{C[0,T]}}\int_{0}^{t}{\frac{{{\lambda }_{k}}{{m}_{\sigma }}}{1+\mu {{\lambda }_{k}}}{{(t-s)}^{\rho -1}}{{E}_{\rho ,\rho }}\left( -\frac{{{\lambda }_{k}}{{m}_{\sigma }}}{1+\mu {{\lambda }_{k}}}{{(t-s)}^{\rho }} \right)ds}\le \left| {{\phi }_{k}} \right|+\frac{\left| {{g}_{k}} \right|}{{{\lambda }_{k}}{{m}_{\sigma }}}{{\left\| r(t) \right\|}_{C[0,T]}}.
$$

Hence,
\begin{equation} \label{5.16}
{{\left\| {{u}_{k}}(t) \right\|}_{C[0,T]}}\le \left| {{\phi }_{k}} \right|+\frac{{{C}_{1}}}{{{\lambda }_{k}}{{m}_{\sigma }}}\left| {{g}_{k}} \right|. 
\end{equation}

Considering estimate \eqref{5.16}, inequality \eqref{5.15} can be simplified as
\begin{equation}\label{5.17}
\left| _{{{t}_{1}}}D_{t}^{\rho }{{u}_{k}}(\xi ) \right| \leq {{C}_{1}}\left( 1+\frac{2{{M}_{\sigma }}}{{{m}_{\sigma }}} \right)\left| {{g}_{k}} \right|+\frac{2{{M}_{\sigma }}}{\mu }\left| {{\phi }_{k}} \right|.
\end{equation}

According to \eqref{5.14} and \eqref{5.17}, an upper bound for the expression $J_2$ is given by
$$
{{J}_{2}}\le \frac{{{M}_{\sigma }}}{\mu \Gamma (\rho )\left| F[{{(I+\mu A)}^{-1}}g] \right|}{{\left| {{t}_{1}}-{{t}_{2}} \right|}^{\rho }}\sum\limits_{k=1}^{\infty }{\left[ {{C}_{1}}\left( 1+\frac{2{{M}_{\sigma }}}{{{m}_{\sigma }}} \right)\left| {{g}_{k}} \right|+\frac{2{{M}_{\sigma }}}{\mu }\left| {{\phi }_{k}} \right| \right]\left| F\left[ {{v}_{k}} \right] \right|}.
$$

Considering conditions (2)-(4) of Theorem~\ref{thm5.2} and employing the Cauchy–Schwarz inequality, the above expression can be represented as:
$$
{{J}_{2}}\le \frac{{{C}_{F}}{{M}_{\sigma }}}{\mu \Gamma (\rho )\left| F[{{(I+\mu A)}^{-1}}g] \right|}\left( {{C}_{1}}\left( 1+\frac{2{{M}_{\sigma }}}{{{m}_{\sigma }}} \right){{\left\| g \right\|}_{D(A^{\gamma })}}+\frac{2{{M}_{\sigma }}}{\mu }{{\left\| \phi  \right\|}_{D(A^{\gamma })}} \right){{\left| {{t}_{1}}-{{t}_{2}} \right|}^{\rho }}.
$$

If $\delta_2$ is chosen in the this form 
$$
{{\delta }_{2}}={{\left( \frac{\mu \Gamma (\rho )\left| F[{{(I+\mu A)}^{-1}}g] \right|\varepsilon }{3{{C}_{F}}{{M}_{\sigma }}\left( {{C}_{1}}\left( 1+\frac{2{{M}_{\sigma }}}{{{m}_{\sigma }}} \right){{\left\| g \right\|}_{D(A^{\gamma })}}+\frac{2{{M}_{\sigma }}}{\mu }{{\left\| \phi  \right\|}_{D(A^{\gamma })}} \right)} \right)}^{\frac{1}{\rho }}}
$$
then, for all $t_1, t_2 \in [0,T]$ such that $|t_1 - t_2| < \delta_2$, the following holds:
\begin{equation}\label{5.18}
   J_2 <\frac{\varepsilon}{3}. 
\end{equation}

Third, the term $J_3$ is estimated using estimate \eqref{5.16} and the fact that $\frac{\mu \lambda_k}{1+\mu \lambda_k} < 1$:
$$
{{J}_{3}}\le \frac{\left| \sigma \left( {{t}_{1}} \right)-\sigma \left( {{t}_{2}} \right) \right|}{\mu \left| F[{{(I+\mu A)}^{-1}}g] \right|}\sum\limits_{k=1}^{\infty }{\left[ \left| {{\phi }_{k}} \right|+\frac{{{C}_{1}}}{{{\lambda }_{k}}{{m}_{\sigma }}}\left| {{g}_{k}} \right| \right]\left| F\left[ {{v}_{k}} \right] \right|}.
$$

Under conditions (2)-(4) of Theorem~\ref{thm5.2} and by using the Cauchy–Schwarz inequality, the above expression is reduced to:
$$
{{J}_{3}}\le \frac{{{C}_{F}}}{\mu \left| F[{{(I+\mu A)}^{-1}}g] \right|}\left( {{\left\| \phi  \right\|}_{D(A^{\gamma })}}+\frac{{{C}_{1}}}{{{m}_{\sigma }}}{{\left\| g \right\|}_{D(A^{\gamma })}} \right)\left| \sigma \left( {{t}_{1}} \right)-\sigma \left( {{t}_{2}} \right) \right|.
$$

Since $\sigma(t) \in C[0,T]$, for any $\varepsilon_3>0$, there exists $\delta_3>0$ such that for all $t_1, t_2 \in [0,T]$, if $|t_1 - t_2| < \delta_3$, then
$$
\left| \sigma \left( {{t}_{1}} \right)-\sigma \left( {{t}_{2}} \right) \right|<\varepsilon_3.
$$

By setting 
$$
{{\varepsilon }_{3}}=\frac{\mu \left| F[{{(I+\mu A)}^{-1}}g] \right|}{3{{C}_{F}}\left( {{\left\| \phi  \right\|}_{D(A^{\gamma })}}+\frac{{{C}_{1}}}{{{m}_{\sigma }}}{{\left\| g \right\|}_{D(A^{\gamma })}} \right)}\varepsilon ,
$$
there exists a corresponding $\delta_3=\delta(\varepsilon_3)$, which implies that 
\begin{equation}\label{5.19}
   J_3 <\frac{\varepsilon}{3}. 
\end{equation}

By virtue of estimates \eqref{5.13}, \eqref{5.18} and \eqref{5.19}, for any $\varepsilon>0$ there exists $\delta = \min\{\delta_1,\delta_2,\delta_3\}$ such that for all $t_1,t_2\in[0,T]$ with $|t_1-t_2|<\delta$, one has
$$
\left|B[r]\left(t_1\right)-B[r]\left(t_2\right)\right| < \varepsilon.
$$

\end{proof}

\begin{lem} \label{lem5.6} The problem \eqref{1.1}--\eqref{1.3} has at least one solution.
\end{lem}
\begin{proof} 
It is well known that the set $U$ is a closed and convex subset of $C[a,b]$, and by Lemma~\ref{lem5.5}, the set $\mathcal{K}$ is relatively compact in $C[a,b]$. Therefore, by Schauder’s fixed point theorem, the operator $B: U \to U$ admits at least one fixed point $r^* \in U$, that is,
$$
r^*(t) = B[r^*](t).
$$

It follows that problems \eqref{1.1}–\eqref{1.3} admit at least one solution.

\end{proof}

\subsection{Uniqueness of solution for the inverse problem}  The uniqueness of the solution to the inverse problem \eqref{1.1}–\eqref{1.3} is established in this subsection.

Suppose that the inverse problem \eqref{1.1}–\eqref{1.3} admits two solutions  $\{\tilde{u}(t);r_1(t)\}$ and $\{\bar{u}(t);r_2(t)\}$. Defining $u(t)=\tilde{u}(t) - \bar{u}(t)$, we consider the following problem.
\begin{equation}\label{5.20}
\left\{
\begin{aligned}
    & D_t^{\rho} [u(t) + \mu Au(t)] + \sigma(t) Au(t) = [r_1(t)-r_2(t)]g, \qquad t \in (0,T]; \\
    & u(0) = 0;\\
    & F[u(t)] = 0, \qquad t \in [0,T].
\end{aligned}
\right.
\end{equation}

Using estimate \eqref{5.10}, we arrive at the following inequality:
$$
|r_1(t) - r_2(t)| \le 0.
$$
Therefore, the solution of the inverse problem is unique.

\section{NUMERICAL RESULTS FOR THE FORWARD PROBLEM} Numerical results corresponding to the forward problem \eqref{1.1}-\eqref{1.2} are presented in this section.

Let us consider the problem \eqref{1.1}–\eqref{1.2} in a one-dimensional setting. This approach allows for a simplified analysis of the problem and facilitates the study of its theoretical and numerical properties.

Let $H$ be $L^2(0,1)$ and $Au=-u_{x x}$, $ x \in(0,1)$, with homogeneous Dirichlet boundary condition. Then the operator $A$ has the eigensystem $\left\{(\pi k)^2, \sqrt{2} \sin k \pi x\right\}_{k \in \mathbb{N}}$.

Consider the time-fractional pseudo-parabolic equation:
\begin{equation} \label{6.1}
D_t^{\rho}u(x,t) - \mu D_t^{\rho}u_{xx}(x,t) - \sigma(t) u_{xx}(x,t) = r(t)g(x), \qquad (x,t) \in (0,1) \times (0,T], 
\end{equation}
with the initial condition
\begin{equation}\label{6.2}
u(x,0) = \phi(x), \qquad x\in [0,1],  
\end{equation}
and the boundary condition
\begin{equation}\label{6.3}
u(0,t) = u(1,t)=0, \qquad \phi(0)=\phi(1)=0, \qquad 0 \le t \le T.
\end{equation}

In order to illustrate the numerical solution of the problem \eqref{6.1}–\eqref{6.3}, the functions satisfying the conditions of Theorem~\ref{thm4.3} are selected as
$$
\sigma(t) = 2+\sqrt{t}, \quad 
\phi(x) = \sqrt{2}\sin(\pi x), \quad 
g(x) = \sqrt{2}(1+\pi^2)\sin(\pi x),
$$
$$
r(t) = \frac{16}{3\sqrt{2\pi}}\, t^{3/2}
+ \frac{\sqrt{2}\pi^2}{(1+\pi^2)}\,(2+\sqrt{t})(1+t^2),
$$
with parameters $T=5$, $\mu=1$ va $\rho=\tfrac{1}{2}$.

The exact solution of the problem \eqref{6.1}–\eqref{6.3} is given by
$$
u(x,t) = 2(1+t^2)\sin(\pi x)
$$
which is used to verify the accuracy of the numerical solution.

\subsection{ Numerical method for the forward problem }  
In this section, a finite difference approximation is constructed for the time-fractional pseudo-parabolic equation \eqref{6.1} with the initial condition \eqref{6.2} and boundary conditions \eqref{6.3}.

The spatial interval $[0,1]$ is divided into $N+1$ subintervals with step size $h=\frac{1}{N}$, and the grid points are defined as
$$
x_i=ih, \qquad i=0,1,2,...,N.
$$

Similarly, the time interval $[0,T]$ is divided into $M+1$ points using a graded mesh
$$
t_k=T\left(\frac{k}{M}\right)^r, \quad k=0,1, \ldots, M,
$$
where $r \ge 1$. The time step is denoted by
$$
\tau_k=t_k-t_{k-1}.
$$

Let $u_i^k \approx u(x_i,t_k)$ represent the numerical approximation of the solution. The Caputo fractional derivative $D_t^{\rho} u(x,t)$ is approximated by the non-uniform $L1$ scheme:
$$
    D_t^{\rho}u(x_i,t_k) \approx \frac{1}{\Gamma(2-\rho)} \sum_{j=1}^k d_{k, j}\left(u_i^j-u_i^{j-1}\right),
$$
where, for $1 \leq j \leq k$:
$$
d_{k, j}:=\frac{\left(t_k-t_{j-1}\right)^{1-\rho}-\left(t_k-t_j\right)^{1-\rho}}{\tau_j} .
$$

The second-order spatial derivative is approximated using the central difference formula:
$$
u_{x x}\left(x_i, t_k\right) \approx \frac{u_{i+1}^k-2 u_i^k+u_{i-1}^k}{h^2} .
$$

Similarly,
$$
D_t^\rho u_{x x}\left(x_i, t_k\right) \approx \frac{1}{\Gamma(2-\rho)} \sum_{j=1}^k d_{k, j}\left(\frac{u_{i+1}^j-2 u_i^j+u_{i-1}^j}{h^2}-\frac{u_{i+1}^{j-1}-2 u_i^{j-1}+u_{i-1}^{j-1}}{h^2}\right) .
$$

Substituting these approximations into equation \eqref{6.1} gives
$$
\frac{1}{\Gamma(2-\rho)} \sum_{j=1}^k d_{k, j}\left(u_i^j-u_i^{j-1}\right)-\mu \frac{1}{\Gamma(2-\rho)} \sum_{j=1}^k d_{k, j}\left(\frac{u_{i+1}^j-2 u_i^j+u_{i-1}^j}{h^2}-\frac{u_{i+1}^{j-1}-2 u_i^{j-1}+u_{i-1}^{j-1}}{h^2}\right)
$$
$$
-\sigma^k \frac{u_{i+1}^k-2 u_i^k+u_{i-1}^k}{h^2}=r^k g_i,
$$
for $i=1, ... ,N-1$ and $k=1, ... ,M$.

Separating the terms corresponding to the current time level $k$ and rearranging,
$$
-\left(\frac{\mu d_{k, k}}{\Gamma(2-\rho) h^2}+\frac{\sigma^k}{h^2}\right) u_{i+1}^k +\left(\frac{d_{k, k}}{\Gamma(2-\rho)}+\frac{2 \mu d_{k, k}}{\Gamma(2-\rho) h^2}+\frac{2 \sigma^k}{h^2}\right) u_i^k-\left(\frac{\mu d_{k, k}}{\Gamma(2-\rho) h^2}+\frac{\sigma^k}{h^2}\right) u_{i-1}^k=b_i^k,
$$
where the right-hand side is
$$
b_i^k=\frac{d_{k, 1}}{\Gamma(2-\rho)} u_i^0+\frac{1}{\Gamma(2-\rho)} \sum_{j=1}^{k-1}\left(d_{k, k-j}-d_{k, k-j+1}\right) u_i^{k-j}
$$
$$
-\frac{\mu}{\Gamma(2-\rho) h^2} \sum_{j=1}^{k-1}\left(d_{k, k-j}-d_{k, k-j+1}\right)\left(u_{i+1}^{k-j}-2 u_i^{k-j}+u_{i-1}^{k-j}\right)+r^k g_i.
$$

Collecting all equations for $i=1,...,N-1$, the linear algebraic system is obtained:
\begin{equation}\label{6.5}
   A^ku^k = b^k, 
\end{equation}
where 
$$
u^k=\left(\begin{array}{c}
u_1^k \\
u_2^k \\
\vdots \\
u_{N-1}^k
\end{array}\right) .
$$

The coefficient matrix $A^k$ has a tridiagonal form:
$$
A^k=\left(\begin{array}{ccccc}
c^k & -a^k & 0 & \cdots & 0 \\
-a^k & c^k & -a^k & \cdots & 0 \\
0 & -a^k & c^k & \ddots & \vdots \\
\vdots & \ddots & \ddots & \ddots & -a^k \\
0 & \cdots & 0 & -a^k & c^k
\end{array}\right),
$$
with
$$
a^k=\frac{\mu d_{k, k}}{\Gamma(2-\rho) h^2}+\frac{\sigma^k}{h^2},
\qquad c^k=\frac{d_{k, k}}{\Gamma(2-\rho)}+\frac{2 \mu d_{k, k}}{\Gamma(2-\rho) h^2}+\frac{2 \sigma^k}{h^2}.
$$
The right-hand side vector is
$$
b^k=\left(\begin{array}{c}
b_1^k \\
b_2^k \\
\vdots \\
b_{N-1}^k
\end{array}\right),
$$
where $b_i^k$ is given above.

Since the matrix $A^k$ is tridiagonal, system \eqref{6.5} can be efficiently solved using the Thomas algorithm.

\FloatBarrier
This algorithm describes the numerical procedure for solving the forward problem of the time-fractional pseudo-parabolic equation using a finite difference scheme.
\begin{table}[!htbp] 
\centering
\caption{Numerical algorithm for the forward problem}\label{tab:forward_algorithm}
\begin{tabular}{>{\raggedright\arraybackslash}p{0.95\textwidth}}
\toprule
\textbf{Finite difference scheme for the forward problem} \\
\midrule
1: Set domain $T$, and fractional order $\rho \in (0,1)$;\\
2: Choose number of spatial nodes $N+1$ and time steps $M+1$;\\
3: Specify functions $\sigma(t)$, $\phi(x)$, $g(x)$,  $r(t)$;\\
4: Construct spatial grid $x_i=ih$;\\
5: Construct graded time mesh $t_k=T\left(\frac{k}{M}\right)^r$;\\
6: Initialize $u_i^0 = \phi(x_i)$;\\
7: For $k=1, ... ,M$:\\
\qquad 7.1: Compute L1 weights $d_{k,j}$;\\
\qquad 7.2: Assemble matrix $A^k$ and right-hand side $b_k$;\\
\qquad 7.3: Solve $A^ku^k=b^k$ using Thomas algorithm;\\
\qquad 7.4: Update boundary derivatives $u_x(0,t_k)$, $u_x(1,t_k)$;\\
8: End for.\\
Output: Numerical solution $\left\{u_i^k\right\}_{i, k=0}^{N, M}$.
\end{tabular}
\end{table}

\FloatBarrier
Using the algorithm given in Table~\ref{tab:forward_algorithm}, the numerical and exact solutions were computed and are illustrated in the figure below. The plot corresponds to $N=1000$ and $M=100$:
\begin{figure}[!htbp]
    \centering
    \begin{minipage}{0.45\textwidth}
        \centering
        \includegraphics[height=6cm]{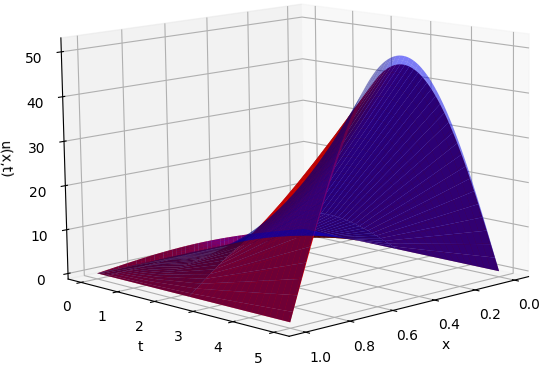}
        \captionof{figure}{Exact (blue surface) and numerical (red surface) solutions of $u(x,t)$.}
        \label{fig:rasm1}
    \end{minipage}
    \hfill
    \begin{minipage}{0.45\textwidth}
        \centering
        \includegraphics[height=6cm]{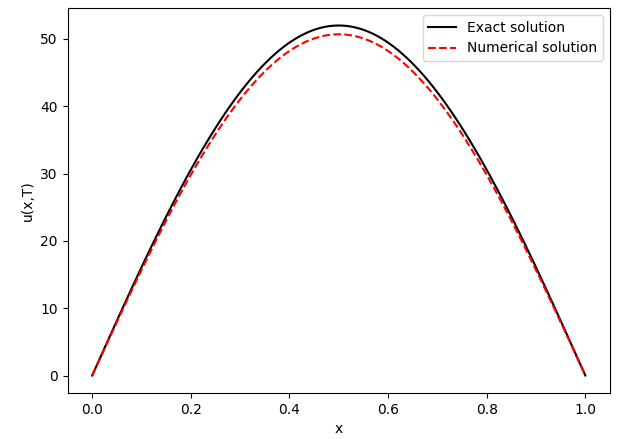}
        \captionof{figure}{Exact and numerical solutions at final time $T$.}
        \label{fig:rasm2}
    \end{minipage}
\end{figure}

\FloatBarrier
\section{CONCLUSION} In this work, forward and inverse problems for a time-fractional pseudo-parabolic equation in a Hilbert space have been investigated. In contrast to the existing literature, we consider a time-dependent coefficient $\sigma(t)$, which significantly generalizes the underlying model. For the forward problem, a spectral representation of the solution is obtained via the Fourier method, and the global existence and uniqueness of the solution are rigorously established. Moreover, in the case where the operator $A$ is a second-order differential operator, a numerical scheme and an efficient computational algorithm are developed. For the inverse problem, we study the determination of a time-dependent source function under an overdetermination condition of the form $F[u(t)] = \Phi(t)$. A distinctive feature of the present work is that the functional $F$ is considered in a general form, which substantially broadens the class of admissible inverse problems. By applying Schauder’s fixed point theorem, we prove the global existence of a solution and establish its
uniqueness.

\section{ACKNOWLEDGEMENTS}
The authors express their thanks to Sh. A. Alimov and Z. A. Sobirov for discussions of the results.

\FloatBarrier

\end{document}